\documentclass[12pt]{amsart}
\usepackage{enumerate,amssymb,amsfonts}
\usepackage{amsthm,color,amsmath}
\usepackage{comment}
\usepackage{hyperref}
\usepackage{mathtools}
\usepackage[utf8]{inputenc} 
\usepackage[T1]{fontenc}

\usepackage{graphicx}
 \usepackage{cite}
\newcommand{\R}{{\mathbb R}}

\newcommand{\N}{{\mathbb N}}

\def\0{{\mathbf 0}}

\newcommand{\e}{\epsilon}
\newcommand{\vp}{\varphi}

\newcommand{\esssup}{\operatornamewithlimits{ess\,sup}}

\newcommand{\ddiv}{\operatorname{div}}

\newcommand{\ra}{\rightarrow}

\newcommand\norm[1]{\Arrowvert {#1}\Arrowvert}

\theoremstyle{plain}
\newtheorem{theorem}{Theorem}[section]

\newtheorem{lemma}[theorem]{Lemma}

\theoremstyle{definition}

\theoremstyle{remark}

\newtheorem{remark}[theorem]{Remark}

\numberwithin{equation}{section}

 \title[Vectorial Linear Transmission]{Lipschitz regularity in vectorial\\ linear transmission problems   }

\author{Alessio Figalli}
\email{alessio.figalli@math.ethz.ch}
\address{Department of Mathematics, ETH Z\"urich,  Raemistrasse 101, 8092 Z\"urich, Switzerland }

\author{Sunghan Kim}
\email{sunghan@kth.se}
\address{Department of Mathematics, KTH Royal Institute of Technology, 100 44 Stockholm, Sweden}

\author{Henrik Shahgholian}
\email{henriksh@kth.se}
\address{Department of Mathematics, KTH Royal Institute of Technology, 100 44 Stockholm, Sweden}
\thanks{S.\ K.\ was supported by postdoctoral fellowship from Knut and Alice Wallenberg Foundation. H.\
Sh.\ was partially supported by Swedish Research Council. A.\ F.\ is supported by the European Research Council under the Grant Agreement No. 721675 “Regularity and Stability in Partial Differential Equations (RSPDE).”}

\begin{document}

\begin{abstract}
We consider vector-valued solutions to a linear transmission problem, and we 
prove that  Lipschitz-regularity on one phase is transmitted to the next phase. More exactly, given a solution $u:B_1\subset \R^n \to \R^m$ to the elliptic system
\begin{equation*}
\ddiv ((A + (B-A)\chi_D )\nabla u) = 0 \quad \text{in }B_1, 
\end{equation*} 
where $A$ and $B$ are Dini continuous, uniformly elliptic matrices, we prove that if $\nabla u \in L^{\infty} (D)$ then $u$ is Lipschitz in $B_{1/2}$. 
A similar result is also derived for the parabolic counterpart  of this problem.

\end{abstract}


\maketitle
\tableofcontents


\section{Introduction}

\subsection{Background}
This paper concerns  optimal regularity  results for vector-valued solutions to linear elliptic systems (and their parabolic counterparts), with free boundaries, for  the so-called transmission problem 
\begin{equation}\label{eq:trans}
\ddiv ((A + (B-A)\chi_D )\nabla u) = 0 \quad \text{in }B_1;
\end{equation} 
see below for  notational specification and exact definitions. 

The  transmission problem has long been under scrutiny and subject to intense study  from various aspects: existence, regularity, geometry of the free boundary, etc. Its importance  has shown to be central in many applications when composite materials are used. To avoid digression from the main mathematical problem, we refer the interested  reader  to two books that cover such applications \cite{Am-Ka}, \cite{Isak-2017}.

 In this paper we introduce yet another type of question, concerning the fine regularity of solutions. Indeed, under rather general assumptions, we prove  that if a solution to this problem is Lipschitz in $D$, then it is Lipschitz in the ball $B_{1/2}$. The proof is  inspired by the approach in       \cite{ALS} and \cite{FS14},  where the authors proved similar results for the scalar  obstacle-type problems.

 Our results can be set in the context of optimal regularity of solutions,  subject to harmonic continuation property (see Section \ref{discuss} for an explanation)  in classical inverse-conductivity problem, as treated in \cite{ACS} (see also  \cite{AL-IS} for the two-dimensional case). Related  results have  been considered in \cite{KLS1}, \cite{KLS2}.
 It needs to be remarked that the techniques from these references do not apply to our setting, since our problem has different prerequisites and   is of a different  nature. Indeed, under  harmonic-continuation-property assumption,  one   uses the well-established monotonicity formula to prove Lipschitz regularity of solutions, as done in   \cite{ACS}. 
 The approach of  \cite{ACS} to prove Lipschitz regularity for solutions   could be extended also to the case of $C^{2}$-continuation-property; see Section \ref{discuss-1} for some explanation. 
Our approach is more general, as we only assume the solution to be Lipschitz ``on one side'', i.e., $\nabla u\in L^\infty(D)$, and the proof applies to linear systems and possibly  to several other equations. We shall discuss this further,  along with other aspects of the problem, in Section~\ref{discuss}.

It is noteworthy that we do not impose any assumption on the regularity of $\partial D$. As for regular boundaries, one can obtain the Lipschitz regularity of $u$ across $\partial D$, without the assumption $\nabla u\in L^\infty(D)$. For instance, in \cite{LN}, it is proved that $\nabla u\in L^\infty$ when $\partial D$ is $C^{1,\alpha}$, and that the derivatives of $u$ are H\"older continuous up to $\partial D$ from each side.

\subsection{Definitions and standing assumptions}

Throughout  the paper, the parameters $n$, $m$, $\lambda$, $\Lambda$, and $\omega$ will be fixed, unless stated otherwise. By $(f)_{z,r}$, we shall denote the average of $f$ over the ball $B_r(z)$, i.e., $$(f)_{z,r} = \frac{1}{|B_r(z)|} \int_{B_r(z)} f(x)\,dx.$$ In addition, we shall simply write $(f)_r$ for $(f)_{0,r}$. In Section \ref{section:proof-parabolic}, we shall follow the usual parabolic terminology: parabolic cubes $Q_r(X) = B_r(x)\times (t-r^2,t)$ with $X = (x,t)\in\R^{n+1}$, $Q_r = Q_r(0)$, the parabolic distance $d_p(X,X') = \sqrt{|x - x'|^2 + (t-t')}$, and the parabolic boundary  $$\partial_p Q_r(X) = (B_r(x)\times \{t-r^2\} )\cup (\partial B_r(x) \times (t-r^2,t)).$$ 

The  following  elliptic system and also its parabolic counterpart (see \eqref{eq:main-para}) are the main equations treated in this paper:
\begin{equation}\label{eq:main}
\ddiv ( (A + (B - A)\chi_D )\nabla u) = 0 \quad \hbox{in } B_1
\end{equation}
where $A = (a_{ij}^{\alpha\beta})_{1\leq i,j\leq m}^{1\leq \alpha,\beta\leq n}$, and $B =  (b_{ij}^{\alpha\beta})_{1\leq i,j\leq m}^{1\leq \alpha,\beta\leq n}$ are coefficient mappings, and $D\subset\R^n$ is an open subset. 
We  say $u$ is a weak solution of \eqref{eq:main} in $B_1$, if $u\in W^{1,2}(B_1;\R^m)$ and 
$$
\int_{B_1} \left(a_{ij}^{\alpha\beta} + \left(b_{ij}^{\alpha\beta}-a_{ij}^{\alpha\beta}\right)\chi_D \right) \partial_\beta u^j \partial_\alpha \vp^i\,dx = 0,
$$
for any $\vp\in W_0^{1,2}(B_1;\R^m)$, where we used summation convention  over repeated indices. 

We specify the conditions on the coefficients $A$ and $B$ as follows:
\begin{enumerate}[(i)]
\item (Ellipticity) There exists a constant $\lambda \in(0,1)$ such that 
\begin{equation}\label{eq:a-ellip}
\begin{split}
\min\left\{\inf_{B_1} a_{ij}^{\alpha\beta} \xi_\alpha^i \xi_\beta^j, \inf_{B_1} b_{ij}^{\alpha\beta} \xi_\alpha^i \xi_\beta^j \right\} \geq \lambda |\xi|^2,
\end{split}
\end{equation}
for any $\xi\in \R^{mn}$. 
\item (Boundedness) With the same $\lambda$ as above,
\begin{equation}\label{eq:a-bd}
\max\left\{ \sup_{B_1} |a_{ij}^{\alpha\beta}|,\sup_{B_1} |b_{ij}^{\alpha\beta}| \right\} \leq \frac{1}{\lambda}, 
\end{equation} 
for any $1\leq i,j\leq m$ and any $1\leq \alpha,\beta\leq n$. 
\item (Regularity) There exist a Dini modulus of continuity\footnote{That is, $\omega:(0,1]\to (0,\infty)$ is a non-decreasing function satisfying $$\lim_{r\to 0}\omega(r) = 0\qquad \text{and}\qquad \int_0^1\frac{\omega(r)}{r}\,dr<\infty.$$} $\omega$ and a constant $\Lambda > 0$ such that $a_{ij}^{\alpha\beta} \in C^{0,\omega}(B_1)$ and 
\begin{equation}\label{eq:a-dini}
\left[ a_{ij}^{\alpha\beta} \right]_{C^{0,\omega}(B_1)} \leq \Lambda,
\end{equation} 
for each $1\leq \alpha,\beta\leq n$ and each $1\leq i,j \leq m$.
\end{enumerate}
Note that we only require $B$ to be {\it bounded measurable}, where $B$ is the matrix coefficient for the domain $D$ where $u$ is assumed to be Lipschitz. 

\subsection{Main Results}

Our main theorem for elliptic system is the following. 

\begin{theorem}\label{theorem:elliptic}
Let $D\subset\R^n$ be an open set, and $A , B : B_1 \ra \R^{n^2m^2}$ satisfy \eqref{eq:a-ellip}, \eqref{eq:a-bd}, and \eqref{eq:a-dini}. Let $u\in W^{1,2}(B_1;\R^m)$ be a weak solution of \eqref{eq:main} in $B_1$, and assume further that $\nabla u\in L^\infty(B_1\cap D)$. Then $u\in W^{1,\infty}(B_{1/2};\R^m)$ and 
\begin{equation}\label{eq:lip}
\norm{\nabla u}_{L^\infty(B_{1/2})} \leq C ( \norm{u}_{L^2(B_1)} + \norm{\nabla u}_{L^\infty(B_1\cap D)}),
\end{equation}
where $C>0$ depends only on $n$, $m$, $\lambda$, $\Lambda$, and $\omega$. 
\end{theorem}

We also prove the parabolic counterpart of the above regularity theory for weak solutions of
\begin{equation}\label{eq:main-para}
\partial_t u = \ddiv ((A + (B-A)\chi_D)\nabla u)\quad\text{in }Q_1 = B_1\times(-1,0), 
\end{equation}
where $A$ and $B$ are now also time-dependent, $\nabla u$ is the spatial gradient of $u$, and $D$ is an open subset in $\R^{n+1}$. We call $u$ a weak solution of \eqref{eq:main-para} in $Q_1$, provided
$$u\in L^\infty((-1,0); L^2(B_1;\R^m)) \cap L^2((-1,0); W^{1,2}(B_1;\R^m)) $$
and
$$
\int_{Q_1} \left(a_{ij}^{\alpha\beta} + \left(b_{ij}^{\alpha\beta}-a_{ij}^{\alpha\beta}\right)\chi_D \right) \partial_\beta u^j \partial_\alpha \vp^i \,dX = \int_{Q_1} u^i \partial_t \vp^i\,dX, 
$$
for any $\vp \in  W^{1,2}((-1,0); L^2(B_1;\R^m))\cap L^2((-1,0); W^{1,2}(B_1;\R^m)) $ with $\vp(\cdot,-1) = \vp (\cdot,0) = 0$ on $B_1$. 

\begin{theorem}\label{theorem:parabolic}
Let $D\subset\R^{n+1}$ be an open set, and let $A, B : Q_1 \ra \R^{n^2m^2}$ satisfy \eqref{eq:a-ellip}, \eqref{eq:a-bd}, and \eqref{eq:a-dini} (with $B_1$ replaced by $Q_1 = B_1\times (-1,0)$). Suppose that $u$ is a weak solution of \eqref{eq:main-para} in $Q_1$ satisfying $\nabla u\in L^\infty(Q_1\cap D)$. Then for any $X,Y\in Q_{1/2}$ with $X\neq Y$, 
\begin{equation}\label{eq:lip-para}
\frac{|u(X) - u(Y)|}{d_p(X,Y)} \leq C \left( \esssup_{t\in(-1,0)} \norm{u(\cdot,t)}_{L^2(B_1)} + \norm{\nabla u}_{L^\infty(Q_1\cap D)} \right),
\end{equation}
where $C>0$ depends only on $n$, $m$, $\lambda$, $\Lambda$, and $\omega$. 
\end{theorem}

\begin{remark}
Theorems \ref{theorem:elliptic} and \ref{theorem:parabolic} can be easily extended to the case when the right hand sides of \eqref{eq:trans} and \eqref{eq:main-para} are replaced with $\ddiv F$, for some Dini continuous mapping $F$. Here we treat the homogeneous right hand side only for the sake of simplicity. We shall leave such a generalization to the interested reader.  
\end{remark}

\subsection{Organisation of the paper}

The paper is organised as follows. In Section \ref{section:sys-lin} and \ref{section:proof-parabolic}, we prove respectively Theorem \ref{theorem:elliptic} and Theorem \ref{theorem:parabolic}. In Section \ref{discuss}, we shall discuss some relevant problems at a heuristic level, and present some open questions for the interested reader. In Appendix, we include some technical lemmas.


\section{Proof of Theorem~\ref{theorem:elliptic}: elliptic case  }\label{section:sys-lin}

As mentioned before, the analysis here follows closely Sections 2.1 and 2.2 in \cite{FS14}. To simplify the exposition, we shall assume, in addition to the assumptions in Theorem \ref{theorem:elliptic}, that 
\begin{equation}\label{eq:Du-Linf-D}
\norm{u}_{L^2(B_1)} + \norm{\nabla u}_{L^\infty(D\cap B_1)}  \leq 1,
\end{equation}
unless stated otherwise. The general case can be recovered by considering 
$$
\tilde u = \frac{ u }{ \norm{u}_{L^2(B_1)} + \norm{\nabla u}_{L^\infty(D\cap B_1)}}.
$$

\begin{lemma}\label{lemma:bmo}
Under the assumption of Theorem \ref{theorem:elliptic} and \eqref{eq:Du-Linf-D}, one has $\nabla u\in BMO(B_{3/4};\R^{mn})$, and $u\in C^\alpha(B_{3/4};\R^m)$ for any $\alpha\in(0,1)$. In particular, for each $z\in B_{1/2}$ and any $r\in(0,\frac{1}{4})$, there exists a vectorial affine function $\ell_{z,r}$ such that 
\begin{equation}\label{eq:Du-bmo}
\int_{B_r(z)} |\nabla u - \nabla \ell_{z,r}|^2\,dx \leq C r^n, 
\end{equation} 
where $C>0$ depends only on $n$, $m$, $\lambda$, $\Lambda$, and $\omega$. 
\end{lemma}

\begin{proof}
Note that, due to the assumptions \eqref{eq:a-bd} and \eqref{eq:Du-Linf-D}, our equation \eqref{eq:main} can be written as
$$\ddiv (A\nabla u) = \ddiv F,\qquad \text{with }F = (A-B)\chi_D \nabla u \in L^\infty(B_1;\R^{mn}).$$ Hence, choosing a vectorial affine function $\ell_{z,r}$ satisfying $\nabla \ell_{z,r} = (\nabla u)_{z,r}$, the conclusion follows by ellipticity regularity theory (e.g., Theorem 4.1 in \cite{Acq92}), thanks to the assumption that $A$ is Dini continuous on $B_1$. 
\end{proof} 

The following lemma is the analogue of Proposition 2.4 in \cite{FS14}.
We need the following definition:
\begin{equation}\label{eq:A}
D_{z,r} = \{ x\in B_1: rx + z\in B_r(z)\cap D\},
\end{equation} 
and $D_r = D_{0,r}$. 

\begin{lemma}\label{lemma:den}
Assume that $\omega(1)\leq\frac{1}{2}$, $z\in B_{1/2}$, $r\in(0,\frac{1}{4})$, and let $\ell_{z,r}$ be as in Lemma \ref{lemma:bmo}. There exist constants $c>1$ and $M>1$, depending only on $n$, $m$, $\lambda$, $\Lambda$, and $\omega$, such that if $|\nabla \ell_{z,r}| \geq M$, then 
\begin{equation}\label{eq:den}
|D_{z,r/2}| \leq \frac{|D_{z,r}|}{2^n} + c\,\omega(r)^{3n}.
\end{equation}
\end{lemma}

\begin{proof}
Throughout the proof, $C$ and $C_p$ will be  universal  constants,  depending only on $n$, $m$, $\lambda$, $\Lambda$, and $\omega$, with $C_p$ further depending on $p$, and they may vary from one appearance to another. With no loss of generality, we can assume $z = 0$.

Fix $r\in(0,\frac{1}{2})$, and let $\ell_r = \ell_{0,r}$ be a vectorial affine function satisfying \eqref{eq:Du-bmo}. In what follows, we shall write 
\begin{equation}\label{eq:ur}
u_r(x) = \frac{u(rx)}{r}.
\end{equation}

Let $v_r$ be the weak solution of 
\begin{equation}
\label{eq:frozen}
\begin{dcases}
\ddiv (A(0) \nabla v_r ) = 0\quad \text{in }B_1,\\
v_r - (u_r - \ell_r) \in W_0^{1,2}(B_1;\R^m). 
\end{dcases}
\end{equation} 
Then the interior gradient estimate for constant elliptic systems, followed by the Poincar\'e inequality, yields 
\begin{equation}\label{eq:Dvr-Linf-1}
\norm{\nabla v_r}_{L^\infty(B_{2/3})}  \leq C\norm{v_r - (v_r)_1}_{L^2(B_1)} \leq C\norm{\nabla v_r}_{L^2(B_1)}
\end{equation}
(recall that $(v_r)_1$ denotes the average of $v_r$ over $B_1$).
Using $v_r - (u_r - \ell_r)\in W_0^{1,2}(B_1;\R^m)$ as a test function in \eqref{eq:frozen}, we obtain 
$$
\begin{aligned} 
\lambda \norm{\nabla v_r}_{L^2(B_1)}^2 &\leq \int_{B_1} a_{ij}^{\alpha\beta}(0) \partial_\alpha v_r^i\partial_\beta v_r^j \,dx = \int_{B_1} a_{ij}^{\alpha\beta}(0) \partial_\alpha v_r^i \partial_\beta (u_r^j - \ell_r^j)\,dx \\
& \leq \frac{C}{\lambda} \norm{\nabla v_r}_{L^2(B_1)} \norm{\nabla (u_r - \ell_r)}_{L^2(B_1)},
\end{aligned} 
$$ 
and consequently 
 $$\norm{\nabla v_r}_{L^2(B_1)} \leq C\norm{\nabla (u_r - \ell_r)}_{L^2(B_1)}.$$ Combining this inequality with \eqref{eq:Dvr-Linf-1}, and then employing the $L^2$-$BMO$ estimate \eqref{eq:Du-bmo} for $\nabla u$, we arrive at 
\begin{equation}\label{eq:Dvr-Linf}
\norm{\nabla v_r}_{L^\infty(B_{2/3})}+\norm{\nabla v_r}_{L^2(B_1)} \leq C.
\end{equation} 
We now observe that the vector-valued function 
$$
w_r  = u_r - \ell_r  - v_r
$$
is a weak solution of 
$$
\begin{cases}
\ddiv ( A_r \nabla w_r) = \ddiv  (F_r + \phi_r) \quad \text{in }B_1,\\
w_r \in W_0^{1,2}(B_1;\R^m),
\end{cases}
$$
where 
\begin{equation}\label{eq:Ar-Fr}
\begin{aligned}
A_r(x) &= A(rx),\quad \phi_r(x) = (A_r(0) - A_r(x))(\nabla \ell_r + \nabla v_r), \\
B_r(x) &= B(rx),\quad\text{and}\quad F_r(x) = (A_r(x)-B_r(x))\chi_{D_r} \nabla u_r.
\end{aligned}
\end{equation}
Recalling \eqref{eq:Dvr-Linf}, we have
\begin{equation}\label{eq:phir-Linf}
\norm{\phi_r}_{L^\infty(B_{2/3})} + \norm{\phi_r}_{L^2(B_1)} \leq  C\omega(r)( | \nabla \ell_r | + 1 ).
\end{equation}  
On the other hand,  the Lipschitz regularity assumption \eqref{eq:Du-Linf-D} on $u|_D$, together with \eqref{eq:a-bd}, imply that for any  $ 1 \leq p < \infty $
\begin{equation}\label{eq:Fr-Lp}
\int_{B_1} |F_r|^p\,dx \leq C_p |D_r|.
\end{equation}
Since $w_r = 0$ on $\partial B_1$, classical energy estimates combined with \eqref{eq:phir-Linf} and \eqref{eq:Fr-Lp} yield
\begin{equation}\label{eq:Dwr-L2}
\begin{split} 
\int_{B_1} |\nabla w_r|^2\,dx &\leq C\int_{B_1} |F_r + \phi_r|^2\,dx \\
&\leq C \left( |D_r| +  \omega(r)^2( | \nabla \ell_r | + 1 )^2\right).
\end{split}
\end{equation}
Next, it follows from the local $L^p$-theory (Theorem 7.2 in \cite{GM12}), along with \eqref{eq:Dwr-L2}, \eqref{eq:phir-Linf}, and \eqref{eq:Fr-Lp}, that for each $p> n$ it holds
\begin{equation}\label{eq:Dwr-Lp}
\begin{split}
\int_{B_{1/2}} |\nabla w_r|^p \,dx &\leq C_p\left( \norm{\nabla w_r}_{L^2(B_{2/3})}^p + \int_{B_{2/3}} |F_r + \phi_r|^p\,dx \right) \\
&\leq C_p \big( |D_r| + ( \omega(r))^p  \left( | \nabla\ell_r | + 1 \right)^p\big).
\end{split}
\end{equation}   
Finally, combining \eqref{eq:Dvr-Linf}, \eqref{eq:Dwr-Lp}, and \eqref{eq:Du-Linf-D}, and using that  $$\nabla (v_r + w_r) = \nabla u_r - \nabla \ell_r,$$ for any $p>n$ we obtain 
\begin{equation}\label{eq:211}
\begin{split}
|D_r\cap B_{\frac{1}{2}}| |\nabla \ell_r|^p  &\leq \int_{D_r\cap B_{\frac{1}{2}}}\bigl(|\nabla v_r|+|\nabla w_r|+|\nabla u_r|\bigr)^p dx\\
& \leq 3^p\int_{D_r\cap B_{\frac{1}{2}}}\bigl(|\nabla v_r|^p+|\nabla w_r|^p+|\nabla u_r|^p\bigr) dx\\
& \leq C_p\big( |D_r| + \omega(r)^p (|\nabla \ell_r| + 1)^p \big).
\end{split}
\end{equation} 
Finally, we choose $p = 3n$ and $M = (2^{2n} C_{3n})^{1/(3n)}$ in the statement of the proposition,
where $C_{3n}$ is the constant appearing in the last line of \eqref{eq:211} with $p=3n$.
In this way, we have by assumption that $|\nabla \ell_r|^{3n} \geq M^{3n} = 2^{2n}C_{3n}$. Hence, dividing by $|\nabla \ell_r|^{3n}$ both sides of \eqref{eq:211}, and using the relation $|D_{r/2}| = 2^n |D_r \cap B_{1/2}|$, we get
$$
|D_{r/2}|\leq \frac{1}{2^n} |D_r| + C \omega(r)^{3n}.
$$ 
This finishes the proof. 
\end{proof}

Now we are ready to prove the main theorem of this section. With Lemma~\ref{lemma:den} at hand, one can proceed as in the proof of Theorem 1.1 in \cite{FS14}, with some modification due to the dependence of $A$ on $x$. 

\begin{proof}[Proof of Theorem~\ref{theorem:elliptic}]
As discussed before, we can assume that $u$ satisfies \eqref{eq:Du-Linf-D}.
Also, up to rescaling, we can assume that $\omega(1)\leq \frac{1}{2}$.

We shall prove that, for every Lebesgue point $z\in B_{1/2}\setminus \overline D$ of $\nabla u\in L^2(B_1)$,  it holds
\begin{equation}\label{eq:lip-re}
|\nabla u(z)| \leq C_0M,
\end{equation}
where $C_0>1$ depends only on $n$, $m$, $\lambda$, $\Lambda$, and $\omega$. Since almost every point in $B_{1/2}\setminus\overline D$ is a Lebesgue point of $\nabla u$, this will conclude the proof.

Without loss of generality, we can assume that $z= 0$. For $r\in(0,\frac{1}{2})$, we denote by $\ell_r$ a vectorial affine function as in Lemma \ref{lemma:bmo}. As in the proof of Theorem 1.1 in \cite{FS14}, we split the argument  into two cases: 
\begin{enumerate}[({Case} 1)]
\item $\liminf\limits_{k\ra\infty}  \left| \nabla \ell_{2^{-k}} \right| < 2M$,
\item $\liminf\limits_{k\ra\infty}  \left| \nabla \ell_{2^{-k}} \right| \geq 2M$,
\end{enumerate}
where $M>1$ is the large constant chosen from Lemma~\ref{lemma:den}. 

In what follows, we shall denote by $C$ a generic constant that depends only on $n$, $m$, $\lambda$, $\Lambda$, and $\omega$, which may vary upon each occasion. 

In Case 1, the result follows immediately from the $L^2$-$BMO$ estimate \eqref{eq:Du-bmo}, and the assumption that the origin is a Lebesgue point of $\nabla u$. 

In Case 2, we define $k_0\in\N$ as
\begin{equation*}
k_0 = \min\{ k \in \N : |\nabla \ell_{2^{-j}}| \geq M \text{ for any }j\geq k \}.
\end{equation*}
In virtue of  Caccioppoli's inequality for \eqref{eq:main} and \eqref{eq:Du-Linf-D}, we know that $\norm{\nabla u}_{L^2(B_{1/2})} \leq C$, so it follows by \eqref{eq:Du-bmo} that $$|\nabla \ell_{2^{-1}}| \leq \norm{\nabla u}_{L^2(B_{1/2})} + C \leq 2C.$$ Hence, by taking $M$ larger if necessary, we can ensure that $k_0\geq 2$, 

By the definition of $k_0$, we have $|\nabla \ell_{2^{-k_0 + 1}}| < M$. Also, \eqref{eq:Du-bmo} implies that $|\nabla \ell_r - \nabla \ell_{r/2}| \leq C$ for any $r\in(0,\frac{3}{5})$. Thus
\begin{equation}\label{eq:Du-avg}
|\nabla \ell_{2^{-k_0}}| \leq C + M. 
\end{equation} 
On the other hand, since 
$$
|\nabla \ell_{2^{-k_0 - j}}| \geq M,\quad\text{for any }j\in\N,
$$ 
we can apply \eqref{eq:den} at each level $r = 2^{-k_0 - j}$ to get
\begin{equation}\label{eq:den-m}
|D_{2^{-k_0 - j}}| \leq C \left( 2^{-jn}+ \sum_{i = 0}^{j-1} 2^{-in} \omega(2^{-k_0 - j + i})^{3n} \right),
\end{equation} 
where we also used $|D_{2^{-k_0}}|\leq |B_1|$. 

Without loss of generality, assume that $u(0) = \ell_r(0) = 0$, let $u_r $ be as in \eqref{eq:ur}, and define
\begin{equation*}
\hat w_r = u_r - \ell_r, 
\end{equation*}
Thanks to \eqref{eq:Du-bmo} and Poincar\'e inequality, we have 
\begin{equation}\label{eq:whr-W12}
\norm{\hat w_r}_{{W^{1,2}}(B_1)} \leq C.
\end{equation}
Moreover, $\hat w_r$ is a weak solution of 
\begin{equation}\label{eq:whr-pde}
\ddiv (A_r \nabla \hat w_r) = \ddiv (F_r + \hat \phi_r) \quad\text{in }B_1,
\end{equation}
where $A_r$ and $F_r$ are as in \eqref{eq:Ar-Fr}, while
$$
\begin{aligned}
\hat \phi_r = (A_r(0) - A_r)\nabla \ell_r. 
\end{aligned}
$$
Thanks to \eqref{eq:a-bd}, \eqref{eq:Du-Linf-D}, \eqref{eq:den-m}, and the scaling relation $|D_r \cap B_{2^{-j}}| = 2^{-jn} |D_{2^{-j}r}|$,  for any integer $j\geq 1$ we obtain
\begin{equation*}
\begin{aligned}
\int_{B_{2^{-j}}} |F_{2^{-k_0}}|^2 \,dx & \leq C |D_{2^{-k_0}} \cap B_{2^{-j}}| \\
& \leq 2^{-jn}C \left( 2^{-jn}+ \sum_{i = 0}^{j-1} 2^{-jn} \omega(2^{-k_0 - j + i})^{3n} \right). 
\end{aligned}
\end{equation*}
In addition, it follows from \eqref{eq:a-dini} and \eqref{eq:Du-avg} that 
$$
\int_{B_{2^{-j}}} |\hat \phi_{2^{-k_0}}|^2\,dx \leq 2^{-jn}C\, \omega(2^{-k_0 - j})^2 (1+M)^2. 
$$
Combining these two estimates together, we arrive at 
\begin{equation}\label{eq:F-L2}
\int_{B_\rho}  |F_{2^{-k_0}} + \hat \phi_{2^{-k_0}}|^2\,dx \leq CM^2\rho^n\psi(\rho)^2,\quad\text{for all }\rho\in (0,1/2),
\end{equation} 
where 
\begin{equation}\label{eq:psi}
\psi(\rho) = \rho^{n/2}  + \left(\rho^n \int_\rho^1 \frac{\omega(\tau)^{3n}}{\tau^{n+1}}\,d\tau\right)^{1/2} + \omega(\rho).
\end{equation} 
Since $\omega$ is a Dini modulus of continuity, it follows from Lemma \ref{lemma:dini-more} (with $\alpha = \frac{3}{2}n > 1$) that $\int_0^{1/2} \rho^{-1}\psi(\rho)\,d\rho < \infty$. In addition, one can easily verify that $\psi(\rho)$ is non-decreasing in $\rho\in(0,\frac{1}{2})$ and $\lim\limits_{\rho\to 0}\psi(\rho) = 0$. Hence, $\psi$ as in \eqref{eq:psi} is also a Dini modulus of continuity. 

Recalling that $\hat w_{2^{-k_0}}$ is a weak solution of \eqref{eq:whr-pde} satisfying \eqref{eq:whr-W12}, one can deduce from \cite[Proposition 2.1, Remark 2.2]{Li17} along with \eqref{eq:F-L2} that 
\begin{equation}\label{eq:q}
\int_{B_\rho} |\hat w_{2^{-k_0}} - \hat\ell |^2 \,dx \leq CM^2\rho^{n+2} \psi_1(\rho)^2, \quad\text{for all }\rho\in(0,1/4),
\end{equation}
for certain modulus of continuity $\psi_1$ depending only on $\psi$, and some vectorial affine function $\hat\ell$ satisfying 
\begin{equation}\label{eq:q2}
|\hat\ell(0)| + |\nabla \hat\ell| \leq CM.
\end{equation}  
In view of \eqref{eq:whr-pde}, we have 
$$
\ddiv (A_{2^{-k_0}} \nabla (\hat w_{2^{-k_0}} - \hat \ell)) = \ddiv (F_{2^{-k_0}} + \tilde\phi_{2^{-k_0}})\quad\text{in }B_1, 
$$
in the weak sense, where 
$$
\tilde\phi_{2^{-k_0}} = \hat\phi_{2^{-k_0}} + (A_{2^{-k_0}}(0) - A_{2^{-k_0}})\nabla \hat\ell. 
$$
Hence, we can deduce from Caccioppoli inequality, \eqref{eq:a-dini}, \eqref{eq:F-L2}, \eqref{eq:q}, and \eqref{eq:q2}, that
$$
\begin{aligned}
 \int_{B_\rho} |\nabla \hat w_{2^{-k_0}} - \nabla \hat\ell|^2 \,dx &\leq C \int_{B_\rho} \left( \frac{|\hat w_{2^{-k_0}} - \hat\ell |^2}{\rho^2}+  |F_{2^{-k_0}} + \tilde\phi_{2^{-k_0}} |^2\right) dx   \\
&\leq CM^2 \rho^n\big( \psi_1(\rho)^2 + \psi(\rho)^2 + \omega(\rho)^2 \big) \\
&\leq CM^2 \rho^n,
\end{aligned} 
$$
for any $\rho\in(0,\rho_0)$, where $\rho_0$ is chosen so as to satisfy $\psi_1(\rho_0)^2 + \psi(\rho_0)^2 + \omega(\rho_0)^2 \leq 1$. Since $\nabla u_{2^{-k_0}} = \nabla \ell_{2^{-k_0}} + \nabla \hat w_{2^{-k_0}}$, we deduce from \eqref{eq:Du-avg},  \eqref{eq:q2}, and the last inequality, that 
$$
 \int_{B_\rho} | \nabla u_{2^{-k_0}}|^2\,dx \leq C \left( |\nabla \ell_{2^{-k_0}}|^2 \rho^n  +  \int_{B_\rho} |\nabla \hat w_{2^{-k_0}}|^2\,dx \right) \leq CM\rho^n,
 $$
for any $\rho\in(0,\frac{1}{2})$. Dividing by $\rho^n$  both sides, letting $\rho\to 0$, and recalling that the origin is a Lebesgue point of $\nabla u$ (and thus that of $\nabla u_{2^{-k_0}}$), we arrive at 
$$
|\nabla u(0)| = |\nabla u_{2^{-k_0}} (0)| \leq CM,
$$
proving \eqref{eq:lip-re} for $z=0$. 

Repeating this argument at any Lebesgue point $z\in B_{1/2} \setminus\overline D$, the proof is finished. 
\end{proof} 


\section{Proof of Theorem~\ref{theorem:parabolic}: parabolic case  }\label{section:proof-parabolic}

This section is concerned with transmission problems of uniformly parabolic systems, 
\begin{equation}\label{eq:trans-para}
\partial_t u = \ddiv ((A + (B- A)\chi_D)\nabla u)\quad \text{in }Q_1, 
\end{equation}
where $A = (a_{ij}^{\alpha\beta})_{1\leq i,j\leq \leq m}^{1\leq \alpha,\beta\leq n}$ and $B = (b_{ij}^{\alpha\beta})_{1\leq i,j\leq \leq m}^{1\leq \alpha,\beta\leq n}$ are assumed to verify \eqref{eq:a-ellip}, \eqref{eq:a-bd}, and \eqref{eq:a-dini}, with $B_1$ replaced by the unit parabolic cube, $Q_1 = B_1\times (-1,0)\subset\R^{n+1}$; in particular, the Dini continuity \eqref{eq:a-dini}, should now be understood in the parabolic terminology, i.e.,
$$
|a_{ij}^{\alpha\beta} (X) - a_{ij}^{\alpha\beta}(Y) |\leq \Lambda\, \omega(d_p(X,Y)), 
$$
for any $X = (x,t), Y= (y,s)\in Q_1$, where $d_p(X,Y) = \sqrt{|x-y|^2 + |t-s|}$ is the parabolic distance between $X$ and $Y$.

Most of the argument follows Section~\ref{section:sys-lin} and \cite{FS15}. We shall focus on the part that requires new ideas, and omit the arguments that can be derived from the previous section with minor modification.

Analogously to the elliptic case, in addition to the assumptions of Theorem~\ref{theorem:parabolic}, we can always suppose that
\begin{equation}\label{eq:u-lip-para}
\esssup_{t\in(-1,0)} \int_{B_1} |u(x,t)|^2\,dx+ \norm{\nabla u}_{L^\infty(D\cap Q_1)} \leq 1.
\end{equation} 

Let us begin with the $\log$-Lipschitz type estimate.  

\begin{lemma}\label{lemma:bmo-para}
There exists a positive constant $C$, depending only on $n$, $m$, $\lambda$, $\Lambda$, and $\omega$, such that the following holds: for each $Z = (z,s)\in Q_{1/2}$ and $r\in(0,\frac{1}{4})$, there exists some time-independent vectorial linear function $\ell_{Z,r}$ for which $|\nabla \ell_{Z,r}| \leq C |\log r|$ and 
\begin{equation}\label{eq:bmo-para}
\sup_{t\in(-r^2 + s,s)} \int_{B_r(z)} |u(x,t) - u(z,s) - \ell_{Z,r}(x) |^2\,dx  \leq Cr^{n+2}. 
\end{equation} 
\end{lemma} 

\begin{proof}
Note that $u$ is  a weak solution of 
\begin{equation*}
\partial_t u - \ddiv (A\nabla u) = \ddiv F\quad\text{in }Q_1, 
\end{equation*}
where $F = (A-B)\chi_D \nabla u$. Due to \eqref{eq:a-bd} and \eqref{eq:u-lip-para},  $\norm{F}_{L^\infty(Q_1)} \leq C$ for some $C>0$, depending only on $n$, $m$, and $\lambda$. Hence, we can apply Lemma \ref{lemma:bmo-para-gen} for each $Z\in Q_{1/2}$. This yields a constant vector $a_Z\in\R^m$, with $|a_Z|\leq C$, and  a time-independent vectorial linear function $\ell_{Z,r}$, for each $r\in(0,\frac{1}{4})$, such that $|\nabla \ell_{Z,r}|\leq C|\log r|$ and 
$$
\esssup_{t\in(-r^2 + s,s)} \int_{B_r(z)} |u(x,t) - a_Z - \ell_{Z,r}(x) |^2\,dx  \leq Cr^{n+2}.
$$
In particular, using the bound $|\nabla \ell_{Z,r}|\leq C |\log r|$ we easily deduce that that $a_Z = u(Z)$ for a.e. $Z\in Q_{1/2}$, and that $|a_Z - a_Y| \leq Cd_p(Z,Y) |\log d_p(Z,Y)|$ for any $Z,Y\in Q_{1/2}$. Thus, after redefining $u$ in a set of null measure if necessary, we can conclude that $u$ is continuous in $Q_{1/2}$, and $a_Z = u(Z)$ for all $Z\in Q_{1/2}$. Due to the continuity, we can also replace $\esssup$ (in $t$) with $\sup$. This finishes the proof. 
\end{proof} 

Define, for each $r\in(0,\frac{1}{4})$, 
\begin{equation}\label{eq:Dr}
D_{Z,r} = \{ (x,t) \in Q_1 : (rx,r^2 t) + Z\in Q_r(Z)\cap D\},
\end{equation}  
and $D_r = D_{0,r}$. We shall prove a geometric decay of the Lebesgue measure of $D_{Z,r}$, provided that $|\nabla \ell_{Z,r}|$ is sufficiently large.  

\begin{lemma}\label{lemma:decay-para}
Assume that $\omega(r)|\log r|\leq\frac{1}{2}$ for all $r\in(0,\frac{3}{4}]$. Let $Z\in Q_{1/2}$ and $r\in(0,\frac{1}{4})$ be given, and let $\ell_{Z,r}$ be as in Lemma~\ref{lemma:bmo-para} with $r\in(0,\frac{1}{4})$. There are some constants $C>0$ and $M>1$, depending only on $n$, $m$, $\lambda$, $\Lambda$, and $\omega$, such that if $|\nabla \ell_{Z,r}|\geq M$, then 
\begin{equation}\label{eq:decay-para}
 |D_{Z,r/2}| \leq \frac{|D_{Z,r}|}{2^{n+2}} + C\omega(r)^{3n+4}. 
\end{equation}
\end{lemma} 

\begin{remark}\label{remark:decay-para}
Note that the assumption $\omega(r)|\log r|\leq\frac{1}{2}$ for all $r\in(0,\frac{3}{4}]$ can always be satisfied with a Dini modulus of continuity $\omega$, after some scaling; see Lemma \ref{lemma:dini}. 
\end{remark}

\begin{proof}
Throughout this proof, we denote by $C$ a generic constant depending on $n$, $m$, $\lambda$, $\Lambda$, and $\omega$, only. For the matter of simplicity, we shall take $Z = ( 0,0)$. The general case will follow the same lines of argument. 

Subtracting a constant vector if necessary, we  assume that $u(0,0) = 0$, and write  
\begin{equation*}
u_r(x,t) = \frac{u(rx,r^2t)}{r}. 
\end{equation*} 
Let $\ell_r = \ell_{0,r}$ be as in Lemma~\ref{lemma:bmo-para}. Note that $\hat w_r = u_r - \ell_r$ is a weak solution of 
\begin{equation}
\label{eq:w}
\partial_t \hat w_r = \ddiv (A_r \nabla \hat w_r + F_r + \hat \phi_r)\quad\text{in }Q_1,
\end{equation}
where 
\begin{equation*}
\begin{aligned}
A_r(x,t) &= A(rx,r^2t),\quad \hat \phi_r(x,t) = (A_r(x,t) - A_r(0,0)) \nabla \ell_r, \\
B_r(x,t) &= B(rx,r^2t),\quad\text{and}\quad F_r(x,t) = (B_r(x,t) - A_r(x,t))\chi_{D_r} \nabla u_r.
\end{aligned}
\end{equation*}
Also by \eqref{eq:bmo-para}, we have 
\begin{equation}\label{eq:whr-L2-para}
\sup_{t\in(-1,0)} \int_{B_1} |\hat w_r(x,t)|^2\,dx  \leq C.
\end{equation}  Recall from Lemma \ref{lemma:bmo-para} that $|\nabla \ell_r| \leq C|\log r|$. Thus, by \eqref{eq:a-dini} and the assumption $\omega(\rho)|\log \rho|\leq\frac{1}{2}$ for all $\rho\in(0,\frac{3}{4}]$, we have 
\begin{equation}\label{eq:phihr-L2-para}
\int_{Q_{3/4}} |\hat \phi_r|^2 \,dX \leq C(\omega(r) |\log r|)^2 \leq C.
\end{equation} 
On the other hand, thanks to \eqref{eq:a-ellip}, \eqref{eq:a-bd}, and \eqref{eq:u-lip-para}, for any $p\geq 1$ it holds 
\begin{equation}\label{eq:Fr-Lp-para}
\int_{Q_{3/4}} |F_r|^p\,dX \leq C_p |D_r|.
\end{equation}
Therefore, it follows from the Caccioppoli inequality for \eqref{eq:w} that 
\begin{equation}\label{eq:Dur-L2-para}
\int_{Q_{3/4}} |\nabla \hat w_r|^2\,dX \leq C.  
\end{equation} 
Consider now the weak solution $v_r$ to
\begin{equation*}
\begin{cases}
\partial_t v_r = \ddiv (A(0,0) \nabla v_r)& \text{in }Q_{3/4}, \\
v_r = \hat w_r (= u_r - \ell_r) &\text{on }\partial_p Q_{3/4}.
\end{cases}
\end{equation*} 
Combining \eqref{eq:Dur-L2-para} and the interior gradient estimate for constant, linear parabolic systems, we deduce that $\nabla v_r \in L^\infty(Q_{2/3})$ and 
\begin{equation}\label{eq:vr-lip-para}
\norm{\nabla v_r}_{L^\infty(Q_{2/3})} \leq C.
\end{equation}
Observe that the auxiliary function
$$
w_r = u_r - \ell_r - v_r = \hat w_r - v_r 
$$
is a weak solution of 
\begin{equation}\label{eq:wr-pde-para}
\begin{cases}
\partial_t w_r = \ddiv (A_r\nabla w_r + F_r + \phi_r) & \text{in }Q_{3/4},\\
w_r = 0 & \text{on }\partial_p Q_{3/4},
\end{cases}
\end{equation} 
where $A_r$ and $F_r$ are as above, while $\phi_r = (A_r - A_r(0,0))(\nabla \ell_r + \nabla v_r)$. 
By \eqref{eq:Dur-L2-para}, \eqref{eq:Fr-Lp-para}, and \eqref{eq:wr-pde-para}, we obtain 
\begin{equation}\label{eq:Dwr-L2-para}
\begin{split}
\esssup_{t\in(-\frac{9}{16},0)} \int_{B_{3/4}} |w_r(x,t)|^2\,dx &+ \int_{Q_{3/4}} |\nabla w_r|^2\,dX \\
& \leq C \int_{Q_{3/4}} (|F_r|^2 +|\phi_r|^2) \,dX\\
& \leq C \big(|D_r| + \omega(r)^2 (|\nabla \ell_r| + 1)^2\big).
\end{split}
\end{equation}
On the other hand, it also follows from \eqref{eq:vr-lip-para} that
\begin{equation}\label{eq:phir-Linf-para}
\norm{\phi_r}_{L^\infty(Q_{2/3})} \leq \omega(r) ( |\nabla \ell_r| + 1 ).
\end{equation} 
Applying the interior $L^p$-theory \cite[Theorem 4.IV]{Cam} to the parabolic system \eqref{eq:wr-pde-para}, and using  \eqref{eq:Fr-Lp-para}, \eqref{eq:Dwr-L2-para}, and \eqref{eq:phir-Linf-para}, we arrive at 
\begin{equation}\label{eq:Dwr-Lp-para}
\begin{split}
\int_{Q_{1/2}} |\nabla w_r|^p\,dX 
&\leq C_p  \left(\int_{Q_{2/3}} (|w_r|^2 + |\nabla w_r|^2)\,dX\right)^{p/2} \\
&\qquad\qquad + C_p \int_{Q_{2/3}} |F_r + \phi_r|^p\,dX   \\
&\leq  C_p \big( |D_r| + \omega(r)^p (|\nabla \ell_r| + 1)^p \big).
\end{split} 
\end{equation}
The rest of the proof can be finished by following the lines of the proof of Lemma~\ref{lemma:den}; we use \eqref{eq:vr-lip-para}, \eqref{eq:Dwr-Lp-para}, and \eqref{eq:u-lip-para} in replacement of \eqref{eq:Dvr-Linf}, \eqref{eq:Dwr-Lp}, and \eqref{eq:Du-Linf-D}, respectively. We omit the details. 
\end{proof} 

Now we are in  position to prove Theorem~\ref{theorem:parabolic}.

\begin{proof}[Proof of Theorem~\ref{theorem:parabolic}]
We can assume that $u$ is normalised, so to satisfy \eqref{eq:u-lip-para}. We shall first prove the Lipschitz regularity of $u$ in space, and then verify the $\frac{1}{2}$-H\"older continuity of $u$ in time. 

As in the proof of Theorem~\ref{theorem:elliptic}, to prove \eqref{eq:lip-para} it suffices to prove that $\nabla u(Z) \leq C_0M$ for all Lebesgue point $Z\in Q_{1/2}\setminus\overline D$ of $\nabla u \in L^2(-1,0;L^2(B_1))$, where $C_0>1$ is a constant depending only on $n$, $m$, $\lambda$, $\Lambda$, and $\omega$. Again, we present the proof with $Z = (0,0)$ for notational convenience. 

Let $\ell_r = \ell_{0,r}$ be as in Lemma~\ref{lemma:bmo-para}. Choosing $M>1$ as in Lemma~\ref{lemma:decay-para}, we are left with the following dichotomy:
\begin{enumerate}[({Case} 1)]
\item $\liminf\limits_{k\ra\infty}  \left| \nabla \ell_{2^{-k}} \right| < 2M$,
\item $\liminf\limits_{k\ra\infty}  \left| \nabla \ell_{2^{-k}} \right| \geq 2M$.
\end{enumerate}
We can handle each case separately, as in the proof of Theorem~\ref{theorem:elliptic}. The argument can be repeated here almost verbatim; as for the proof for the parabolic counterpart to \eqref{eq:q}, we use Lemma \ref{lemma:C1-para} instead of \cite[Proposition 2.1]{Li17}. This proves the Lipschitz regularity of $u$ in space, with estimate 
\begin{equation}\label{eq:lip-space}
\norm{\nabla u}_{L^\infty(Q_{1/2})} \leq C_0M \leq C, 
\end{equation} 
where $C$ is a constant depending only on $n$, $m$, $\lambda$, $\Lambda$, and $\omega$. 

To show the $\frac{1}{2}$-Lipschitz regularity of $u$ in time, let $Z=(z,s) \in Q_{1/4}$, $r\in(0,\frac{1}{4})$ be arbitrary, and choose $\ell_{Z,r}$ as a time-independent vectorial linear function satisfying Lemma~\ref{lemma:bmo-para}. In what follows, we shall write by $C$ a large constant that may differ at each occasion, yet depends only on $n$, $m$, $\lambda$, $\Lambda$, and $\omega$.

Following the derivation of \eqref{eq:Dur-L2-para}, we obtain that
$$
\esssup_{t\in(-r^2 + s,s)}\int_{B_r(z)} |\nabla u(x,t) -\nabla \ell_{Z,r} |^2\,dx \leq C r^n, 
$$
and thus, along with \eqref{eq:lip-space}, we have 
$$
|\nabla \ell_{Z,r}|^2 \leq C + \frac{C}{r^n}\esssup_{t\in(-r^2 + s,s)} \int_{B_r(z)} |\nabla u(x,t)|^2\,dx \leq C. 
$$
Utilising this inequality in \eqref{eq:bmo-para}, we obtain that
$$
\begin{aligned} 
\sup_{t\in(-r^2 + s,s)}\int_{B_r(z)} | u(x,t) - u(z,s) |^2 \,dx  &\leq Cr^{n+2} + C\int_{B_r(z)} |\ell_{Z,r}(x)|^2\,dx \\
& \leq Cr^{n+2}. 
\end{aligned}
$$
Since the choice of $Z\in Q_{1/4}$ and $r\in(0,\frac{1}{4})$ was arbitrary, we conclude that for any $X,Y \in Q_{1/4}$, 
$$
|u(X) - u(Y)| \leq C d_p (X,Y),
$$
proving the $\frac{1}{2}$-Lipschitz regularity in time as well. 
\end{proof}


\section{Discussions and future directions}\label{discuss}

\subsection{Optimal regularity of solutions} \label{discuss-1}
In this section, we shall discuss the scalar case, although the whole discussion carries over to  the vectorial case. 

The question of  transition of regularity from one phase to another phase for solutions to  (elliptic/parabolic) equations has a central role in the analysis of free boundary problems. Although such questions arise in many applications, the  mere mathematical  point of view is of wide interest among people in PDE/FBP.
They are central in studying a larger class of equations that do not have variational or constrained formulation, as pointed out by two of the current authors in \cite{FS14}.

A question that  appears in  potential theory (and mostly known in scalar case) is the so-called harmonic continuation property. 
To explain this, let $D$ be a given domain in $\R^n$, and let $\sigma_{\partial D}$ denote the surface measure. Consider the 
 single layer potentials\footnote{We assume $\partial D$ has some a priori regularity such that the single layer potential is well defined.}
  $U^{\partial D} (x) = F \star d\sigma_{\partial D}$, where  ``$\star$'' denotes convolution, and $F$ is the (normalised)  fundamental solution of the Laplace operator, so that 
$\Delta U^{\partial D} = - d\sigma_{\partial D}$ in the sense of distributions.
We say $\partial D$ has the harmonic continuation property  near $z \in \partial D$ if there exist $r>0$ and
 a harmonic function $h$ in $B_r(z)$ such    that $U^{\partial D}=h$ in $D \cap B_r (z)$. 
   
   For analytic boundaries, this property holds true due to Cauchy-Kowalevskaya theorem. This is a consequence of  the fact that  one can solve 
   $\Delta v = 0$ in $D \cap B_r (z)$ with Cauchy-data $v=0$, and $|\nabla  v | = 1$ on $\partial D \cap B_r(z)$.    Since  $\Delta U^{\partial D} = \Delta v =  - d\sigma_{\partial D}$,    the function  $h:=  U^{\partial D} - v\chi_D $ is harmonic in $B_r(z)$ and satisfies $h= U^{\partial D} $ in $D^c \cap B_r (z)$; thus, $\partial D$  has the harmonic continuation property near $z \in \partial D$.

Suppose now  $\partial D$ has  the harmonic continuation property close to a boundary point $z\in \partial D$, where  $D$ is given with no a priori regularity assumption for its boundary. 
The question that arise is: ``How regular is the boundary $\partial D \cap B_r (z)$?''
 To study this question, one may (and probably should) start with a simpler question, namely, finding the optimal regularity of $v = U^{\partial D} - h$ in $B_{r/2}(z)$, where $h$ is the harmonic function in $B_r(z)$ mentioned above.
 This amounts to finding  the optimal regularity of 
 $U^{\partial D}$ in $B_{r/2} (z) \setminus D $, given that $\partial D$ has harmonic continuation property.
 
  In \cite{ACS} the authors consider this problem in scalar case for  Lipschitz domains by setting
   $v= U^{\partial D} -  h $, so that it  satisfies $\Delta v= - d\sigma_{\partial D}$ and $v=0$ in $D$. They prove, using a suitable monotonicity formula,  that $v$ is uniformly Lipschitz  in $ B_{r/2} (z) \setminus D$. 

The above regularity question for the single layer potential  is directly connected to the transmission problem studied in this paper. Indeed, for Lipschitz domains one can express solutions to the transmission problem through integral operators, using  layer potentials; see \cite{EFV} (scalar case) or \cite{Am-Ka} (vectorial case). However, it is unknown to us how the Lipschitz regularity assumption on $\partial D$ can be weakened to allow this reformulation. This remains an interesting question  to answer.

Our result in this paper indicates  that, if we can  rephrase the question in terms of the transmission problem \eqref{eq:trans}, then  the single layer potential  $U^{\partial D}$, with $D$ having harmonic continuation property, should be uniformly Lipschitz in $B_{r/2} (z)$.
As pointed out in the introduction, the method of \cite{ACS} works well  if we replace the harmonic continuation property with  $C^{2}$ continuation of the single layer potential, in Lipschitz domains.\footnote{The proof of \cite{ACS} uses the well-known ACF-monotonicity formula, in the harmonic continuation case. For $C^2$-continuation case one can use Caffarelli-Jerison-Kenig monotonicity formula, \cite{CJK}.} 
A natural question is how far one can stretch this relaxation of regularity. 
Our result indicates  that if  $U^{\partial D}$ is  uniformly Lipschitz in $D$, then  this Lipschitz regularity can be transmitted across the boundary. This naturally is true  across regular  boundary points, and preserves the  uniform Lipschitz-norm up to a multiplicative constant, in a neighborhood the boundary. This neighborhood, however,  may possibly become smaller as we come closer to a non-smooth boundary point.
The tantalising question that arises is what {\it a priori} conditions (if any) one should impose on  $\partial D$  to guarantee 
  the transmission of Lipschitz regularity across the boundary for the single layer potential.



 

We shall now formulate two  questions that might be of interest to  readers.



\bigskip
\noindent
{\bf Question 1:} Can one generalise our results to the setting of singular/degenerate operators, such as the $p$-Laplacian? 

\bigskip
\noindent
{\bf Question 2:} Consider nonlinear transmission systems, 
$$
\ddiv (A(\nabla u)\chi_{D^c} + B(\nabla u)\chi_D) = 0,
$$
where both $A$ and $B$ are strongly elliptic, nonlinear operators. It is well-known that nonlinear systems do not have Lipschitz solutions, in general, even if $A = B$ (so the system is homogeneous) and the dependence on $\nabla u$ is smooth. This remains true even for  minimisers  of  a nonlinear functional, see \cite{SY}. It is also known that the boundary regularity fails for nonlinear systems, even if the boundary data is smooth, see e.g., \cite{Gia78}. 
However, if we assume that $u$ is Lipschitz up to $\partial D$, then the Lipschitz regularity may have some chances of propagating to the other side, in some small neighborhood, depending on the geometry of $\partial D$. This is because the governing system yields a matching condition of the normal derivatives of $u$ on $\partial D$: formally, 
$$
A_i^\alpha (\nabla u|_{D^c})\nu_\alpha + B_i^\alpha (\nabla u|_D)\nu_\alpha = 0,
$$
whenever the outward normal $\nu$ is defined on $\partial D$. This may leave us in a better situation than a Dirichlet boundary problem, since for the latter problem the normal derivatives of the solution does not need to match those of the boundary data.

For instance, let $\partial D$ be a hyperplane. Then from the assumptions that $u$ is Lipschitz up to $\partial D$ from $D$, and that the equation yields a matching condition of the normal derivative of $u$ on $\partial D$, it is reasonable in Question 2 to expect the propagation of the Lipschitz regularity to the other side. 

On the other hand, if $\partial D$ has a cusp so that $D$ does not have positive density at a point on $\partial D$, then the nice information from $D$ may lose its effect, and the nonlinearity of the operators in Question 2 does not supplement the loss of information. More precisely, in the blowup regime the limit solution of $u$ will solve $\ddiv A(\nabla u_0) = 0$ everywhere (recall that $A$ is the governing operator in the region $D^c$). Unlike the case of linear systems, the blowup limit $u_0$ may fail to be Lipschitz, so this strategy cannot give any regularity improvement for the original solution $u$. 

This discussion shows that there is still much to explore for the case of nonlinear systems regarding the propagation of the Lipschitz regularity, and we leave this problem open for the future.


%
%

\subsection{Regularity of the free boundary} \label{discuss-2}

In this section we want to discuss the challenging question of regularity of $\partial D$.
For scalar case, the authors in \cite{ACS} study the regularity of  those part of $\partial D$ where the solution does not degenerate; i.e., behaves ``linearly''. They prove that,  under  a priori Lipschitz regularity assumption or a flatness and $\epsilon$-monotonicity of the solution (in a cone of directions), the free boundary is  $C^{1,\alpha}$. 

Still in the scalar case, when $D$ is given by a level set, the authors in \cite{John-Hayk} prove that flat points are almost everywhere with respect to the measure $\Delta u^+$ (in their setting, this is a positive measure whose support is of $\sigma$-finite ($n-1$)-dimensional Hausdorff measure). 

The methods in both \cite{ACS} and \cite{John-Hayk} can be carried out in our setting for the scalar case, under suitable assumptions on the interface $\partial D$. For instance, if $u = \ell$ in $D$ for some affine function $ \ell$, then $u -  \ell$ is essentially the same as in these papers, provided that $u -  \ell$ is non-degenerate across $\partial D$. One may also be able to generalise this by replacing $ \ell$ with some $f\in C^{1,\alpha}(B_1)$. However, the methods in both \cite{ACS} and \cite{John-Hayk} cannot be extended to the systems, since all the techniques are based on maximum/comparison principles. 

In the case of systems, the regularity theory for free boundary problems is wide open, despite its importance. Some essential techniques, such as comparison principles and monotonicity formulas, which are well established for scalar problems, tend to fail for systems in general. Therefore, one has to come up with a new technique to analyze vectorial free boundary problems. In this direction, it will also be interesting to see if one can recover the regularity theory for scalar free boundary problems with energy methods only, and then carry it over to systems. We shall not discuss this issue in more depth, as it goes beyond the scope of this paper.

\appendix

\section{Technical tools}

Let us begin with some lemmas for Dini moduli of continuity. Recall that $\omega$ is said to be a Dini modulus continuity, if $\omega:(0,1]\to (0,\infty)$ is a non-decreasing function satisfying $\lim\limits_{r\to 0} \omega(r) = 0$, and $\int_0^1 \frac{\omega(r)}{r}\,dr < \infty$. 

\begin{lemma}\label{lemma:dini}
If $\omega$ is a Dini modulus of continuity, then $\lim\limits_{r\to 0} \omega(r)\log\frac{1}{r} = 0$. 
\end{lemma}

\begin{proof}
Let $\delta>0$ be arbitrary. Then from the Dini condition, there exists some $r_1\in(0,\frac{1}{2})$ such that for any $r\in(0,r_1)$, 
$$
\delta > \int_r^{r_1} \frac{\omega(s)}{s}\,ds \geq \omega(r) \int_r^{r_1} \frac{ds}{s} = \omega(r) \log \frac{1}{r} - \omega(r) \log \frac{1}{r_1} >0. 
$$
Now we choose a sufficiently small $r_2\in(0,r_1)$ such that $\omega(r) \log \frac{1}{r_1} < \delta$ for all $r\in(0,r_2]$. Then 
$$
0< \omega(r)\log \frac{1}{r_1} < \omega(r) \log \frac{1}{r} < \delta + \omega(r)\log \frac{1}{r_1} < 2\delta, 
$$
for all $r\in(0,r_2)$, proving that $\lim\limits_{r\to 0} \omega(r)\log \frac{1}{r} = 0$.
\end{proof} 

The next lemma is used to prove that $\psi$ as in \eqref{eq:psi} is a Dini modulus of continuity. 

\begin{lemma}\label{lemma:dini-more}
Let $\omega$ be a Dini modulus of continuity. Then for any $\alpha > 1$, 
$$
\int_0^1 r^{\frac{\alpha}{2} - 1} \sqrt{ \int_r^1 \frac{\omega^{2\alpha}(\rho)}{\rho^{\alpha+1}}\,d\rho}\,dr < \infty. 
$$
\end{lemma}

\begin{proof}
According to Lemma \ref{lemma:dini}, there is some $r_0\in(0,1]$ such that $\omega(r) \log \frac{2}{r} \leq 1$ for all $r\in(0,r_0)$. To simplify the notation, we shall assume, without loss of any generality, that $r_0 = 1$. 

Write $\theta(r) = \sqrt{ r^\alpha \int_r^1 \frac{\omega(\rho)^{2\alpha}}{\rho^{\alpha+1}}\,d\rho}$, and let $\e \in (0,\frac{1}{2})$ be arbitrary. By H\"older inequality, 
$$
\left(\int_\e^1 \frac{\theta(r)}{r}\,dr\right)^2 \leq \int_\e^1 \frac{1}{r|\log\frac{r}{2}|^\alpha}\,dr \int_\e^1 \frac{\theta^2(r)|\log\frac{r}{2}|^\alpha}{r}\,dr.
$$
The first integral on the right hand side is bounded uniformly for $\e\in(0,1)$, since by assumption $\alpha>1$. Hence, it suffices to prove the boundedness of the second integral. 

By the Fubini theorem, and integration by parts
$$
\begin{aligned}
\int_\e^1 \frac{\theta^2(r)|\log\frac{r}{2}|^\alpha}{r}\,dr &= \int_\e^1  \frac{\omega^{2\alpha}(\rho)}{\rho^{\alpha+1}} \int_\e^\rho r^{\alpha-1} |\log(r/2)|^\alpha \,dr \,d\rho \\
& \leq C \int_\e^1 \frac{\omega^{2\alpha}(\rho) |\log\frac{\rho}{2}|^\alpha }{\rho}\,d\rho \\
&\leq C \int_\e^1 \frac{d\rho}{\rho |\log \frac{\rho}{2}|^{\alpha}},
\end{aligned}
$$
where $C>1$ is a constant depending only on $\alpha$. Since $\alpha>1$, the integral on the rightmost side is bounded uniformly for all $\e\in(0,\frac{1}{2})$, from which the assertion of the lemma follows immediately. 
\end{proof} 

In what follows, we shall present an interior $BMO$-type estimate and $C^1$-estimate, for the spatial gradients for weak solutions to linear parabolic systems with Dini coefficients. Although these estimates are well understood by experts, we present the proofs for the sake of  convenience for non-expert  readers.

Let us begin with an interior log-Lipschitz estimate.  

\begin{lemma}\label{lemma:bmo-para-gen}
Let $A \in L^\infty((-1,0);L^2(B_1;\R^{n^2m^2}))$ satisfy \eqref{eq:a-ellip}, \eqref{eq:a-bd}, and \eqref{eq:a-dini}, with a modulus of continuity $\omega$ verifying the Dini condition, and let  $F\in L^2((-1,0);L^2(B_1;\R^m))$ be given. Suppose that $u$ is a weak solution of $\partial_tu-\ddiv (A \nabla u) = \ddiv F$ in $Q_1$. Then there exists a constant $C>1$, depending only on $n$, $m$, $\lambda$, $\Lambda$, and $\omega$, such that the following holds: if $\int_{Q_r} |F|^2 \,dX\leq r^{n+2}$ for all $r\in(0,1)$, and $\int_{B_1} |u(x,t)|^2\,dx\leq 1$ for a.e.\ $t\in(-1,0)$, then for each $r\in(0,\frac{3}{4})$, there exist a vectorial time-independent linear function $\ell_r$ and a vector $a\in\R^m$, independent of $r$, such that $|a|\leq C$, $|\nabla \ell_r|\leq C|\log r|$, and 
$$
\esssup_{t\in(-r^2,0)} \int_{B_r} |u(x,t) - a - \ell_r(x)|^2 \,dx  \leq Cr^{n+2}.
$$
\end{lemma}

\begin{proof} 
The proof involves standard approximation techniques. 

According to Lemma \ref{lemma:dini}, $\omega$ satisfies $\lim\limits_{r\to 0} \omega(r)\log \frac{1}{r} = 0$. For this reason, after suitable scaling argument, it suffices to prove the following claim: there are some small positive constants $\mu$ and $\eta$, depending only on $n$, $m$, $\lambda$, $\Lambda$, and $\omega$, such that if, in addition to the assumptions in the statement, for all $r\in(-1,0)$, 
\begin{equation}\label{eq:bmo-para-asmp}
\omega(r)|\log r| \leq\eta\quad \text{and}\quad \int_{Q_r} |F|^2\,dX \leq \eta^2 r^{n+2},
\end{equation} 
then for each $k=1,2,\cdots$, there exists a time-independent vectorial affine function $\ell_k$ such that 
\begin{equation}\label{eq:bmo-para-re1}
|\ell_k (0) - \ell_{k-1}(0)| \leq c\mu^k, \quad |\nabla \ell_k| \leq ck
\end{equation} 
and 
\begin{equation}\label{eq:bmo-para-re}
\esssup_{t\in(-\mu^{2k},0)} \int_{B_{\mu^k}} |u(x,t) - \ell_k(x)|^2\,dx \leq \mu^{k(n+2)},
\end{equation}
where $c>1$ depends only on $n$, $m$ and $\lambda$.

In what follows, $c$ will be a constant depending only on $n$, $m$, and $\lambda$, and $C$ will be a constant depending further on $\Lambda$ and $\omega$. These constants may differ at each appearance.  

Let $\mu\in(0,\frac{3}{4})$ be a sufficiently small number, to be determined, and suppose that we have found, for some integer $k\geq 0$, a time-independent vectorial affine function $\ell_k$, for which \eqref{eq:bmo-para-re} holds; note that for $k = 0$, we can simply choose $\ell_0 = 0$ so the initial case is satisfied. 

Define 
$$
u_k (x,t) = \frac{u(\mu^k x,\mu^{2k} t) - \ell_k(\mu^k x)}{\mu^k}.
$$
Then  $u_k$ is a weak solution of 
\begin{equation}\label{eq:uk-sys}
\partial_t u_k - \ddiv (A_k \nabla u_k) = \ddiv F_k \quad\text{in }Q_1,
\end{equation} 
where 
$$
\begin{aligned}
A_k(x,t) &= A(\mu^k x, \mu^{2k} t),\\
F_k(x,t) &= F(\mu x,\mu^{2k} t) + (A(\mu^k x, \mu^{2k}t)- A(0,0))\nabla \ell_k. 
\end{aligned}
$$
Also, 
by \eqref{eq:bmo-para-re} and the Caccioppoli inequality, $u_k$ satisfies 
\begin{equation}\label{eq:uk-L2-Linf}
\esssup_{t\in(-1,0)}\int_{B_1} |u_k(x,t)|^2\,dx + \int_{Q_{3/4}} |\nabla u_k|^2\,dX \leq c.
\end{equation} 
Clearly, $A_k$ satisfies the same structure conditions \eqref{eq:a-ellip}, \eqref{eq:a-bd}, and \eqref{eq:a-dini}. On the other hand, by \eqref{eq:bmo-para-asmp} and \eqref{eq:bmo-para-re1}, we can deduce that 
\begin{equation}\label{eq:Fk-Linf-L2}
\begin{split}
\int_{-1}^0\int_{B_1} |F_k|^2\,dx\,dt &\leq 2\eta^2 + (n^2m^2\Lambda k\omega(\mu^k))^2 \leq 4\eta^2, 
\end{split}
\end{equation}
provided that we choose $\mu$ so as to satisfy 
\begin{equation}\label{eq:mu}
|\log \mu| \geq n^2m^2\Lambda.
\end{equation} 
Consider the weak solution
 $v_k$  to
\begin{equation}\label{eq:vk-sys}
\begin{cases}
\partial_t v_k = \ddiv (A_k(0,0)\nabla v_k)  & \text{in }Q_{3/4},\\
v_k = u_k & \text{on }\partial_p Q_{3/4}.
\end{cases}
\end{equation}
Let $\beta_\e\in C_0^\infty(\R)$ be a mollifier. Then we can use $\beta_\e \ast (v_k - u_k)$ as the test function to \eqref{eq:vk-sys}. Due to \eqref{eq:uk-L2-Linf}, we obtain 
$$
\esssup_{t\in(-\frac{9}{16}+\e,-\e)} \int_{B_{3/4}} |\beta_\e\ast v_k|^2 \,dx + \int_{-\frac{9}{16}+\e}^{-\e} \int_{B_{3/4}} |\nabla (\beta_\e\ast v_k)|^2\,dx\,dt \leq c. 
$$ 
Hence, letting $\e\to 0$, we arrive at 
\begin{equation}\label{eq:vk-W12}
\esssup_{t\in(-\frac{9}{16},0)} \int_{B_{3/4}} |v_k(x,t)|^2 \,dx + \int_{Q_{3/4}} |\nabla v_k|^2\,dX \leq c. 
\end{equation} 
Thus, by interior $C^{2}_x$ estimates for constant linear parabolic systems, we can find some time-independent vectorial affine function $\hat\ell_k$ such that 
\begin{equation}\label{eq:vk-C1a-1}
|\ell_k(0)| + |\nabla \ell_k|\leq c,
\end{equation} 
and
\begin{equation}\label{eq:vk-C1a}
\sup_{t\in(-r^2,0)} \int_{B_r} |v_k (x,t) - \hat\ell_k(x)|^2\,dx \leq c r^{n+4}.
\end{equation} 
On the other hand, subtracting \eqref{eq:uk-sys} from \eqref{eq:vk-sys}, and then using $\beta_\e \ast (v_k - u_k)$ as the test function to the resulting system (with $\beta_\e$ being the mollifier as above), thanks to \eqref{eq:a-dini}, \eqref{eq:bmo-para-asmp}, \eqref{eq:uk-L2-Linf}, and \eqref{eq:vk-W12}, we deduce that
\begin{multline}
 \esssup_{t\in(-\frac{9}{16}+\e,-\e)} \int_{B_{3/4}} |\beta_\e\ast (v_k  - u_k)|^2\,dx \\
\leq c \int_{Q_{3/4}} ( |F_k|^2 + \omega(\mu^k)  |\nabla v_k|^2)\,dX \leq c\eta^2. 
\end{multline} 
Letting $\e\to 0$, and combining the resulting expression with \eqref{eq:vk-C1a}, yields 
\begin{equation}\label{eq:vk-uk-L2-Linf}
\esssup_{t\in(-\mu^2,0)} \int_{B_\mu} |u_k (x,t) - \hat\ell_k(x) |^2\,dx \leq c(\mu^{n+4} + \eta^2) \leq \mu^{n+2}, 
\end{equation} 
provided that we first choose $\mu$ sufficiently small so that both \eqref{eq:mu} and $c\mu^{n+4} \leq \frac{1}{2} \mu^{n+2}$ hold, and then select $\eta$ accordingly so that $c\eta^2 \leq \frac{1}{2} \mu^{n+2}$. Clearly, $\mu$ and $\eta$ depend only  on $n$, $m$, $\lambda$, $\Lambda$, and $\omega$. 

To this end, we define 
$$
\ell_{k+1} (x) = \ell_k (x) + \mu^k \hat\ell_k \left(\frac{x}{\mu^k}\right),
$$
which is again a vectorial affine function. In view of \eqref{eq:bmo-para-asmp}, \eqref{eq:vk-C1a-1}, \eqref{eq:vk-C1a}, and \eqref{eq:vk-uk-L2-Linf}, this proves the induction hypotheses \eqref{eq:bmo-para-re1} and \eqref{eq:bmo-para-re} with $k$ replaced by $k+1$. The proof is now finished by the induction principle. 
\end{proof} 

Next, we establish  an interior $C^1$ estimate. 

\begin{lemma}\label{lemma:C1-para}
Under the same setting as in Lemma \ref{lemma:bmo-para}, there exists a constant $C>1$, depending only on $n$, $m$, $\lambda$, $\Lambda$, and $\omega$, such that the following holds: if $\int_{Q_r} |F|^2\,dx \leq \omega(r) r^{n+2}$ for all $r\in(0,1)$, and $\int_{B_r} |u(x,t)|^2\,dx \leq 1$ for a.e.\ $t\in(-1,0)$, then there exists a time-independent vectorial affine function $\ell$ such that $|\ell(0)| + |\nabla \ell| \leq C$, and 
\begin{equation*}
\esssup_{t\in(-r^2,0)} \int_{B_r} |u(x,t) - \ell|^2\,dx \leq Cr^{n+2}\omega_1(r), 
\end{equation*} 
where $\omega_1$ is a modulus of continuity depending only on $\omega$. 
\end{lemma}

\begin{proof}
The proof follows   essentially the same lines as that   of  Lemma \ref{lemma:bmo-para}, and it is  omitted. 
\end{proof}



\begin{thebibliography}{9999999}


\bibitem[AM13]{John-Hayk}
J. Andersson and H. Mikayelyan,
{\it The zero level set for a certain weak solution, with applications to the Bellman equations},
Trans. Amer. Math. Soc. {\bf 365} (2013), 2297--2316. 

\bibitem[Acq92]{Acq92}
P. Acquistapace,
{\it On $BMO$ regularity for linear elliptic systems},
Ann. Mat. Pura Appl. (4) {\bf 161} (1992), 231--269. 

\bibitem[ALS13]{ALS}
J. Andersson, E. Lindgren and H. Shahgholian,
{\it Optimal regularity for the no-sign obstacle problem},
Comm. Pure Appl. Math. {\bf 66} (2013), 245--262.

\bibitem[AI96]{AL-IS} 
G. Alessandrini and V. Isakov, 
{\it Analyticity and uniqueness for the inverse conductivity problem},
Rend. Instit. Mat. Univ. Trieste {\bf 28} (1996), 351--369.

\bibitem[AK07]{Am-Ka} 
H. Ammari and H. Kang, 
Polarization and moment tensors,
With applications to inverse problems and effective medium theory. Applied Mathematical Sciences, 162. Springer, New York, 2007. x+312 pp. 


\bibitem[ACS01]{ACS} 
I. Athanasopoulos, L. A. Caffarelli and S. Salsa, 
{\it The free boundary in an inverse conductivity problem},
J. Reine Angew. Math. {\bf 534} (2001), 1--31.

\bibitem[CJK02]{CJK} 
L. A. Caffarelli,  D. Jerison and C. E.  Kenig,
{\it Some new monotonicity theorems with applications to free boundary problems},
Ann. of Math. {\bf 155} (2002), 369--404. 

\bibitem[CDS18]{CDS18}
L. A. Caffarelli, D. De Silva and O. Savin,
{\it Two-phase anisotropic free boundary problems and applications to the Bellman equation in 2D},
Arch. Rational Mech. Anal. {\bf 228} (2018), 477--493.

\bibitem[Cam81]{Cam}
S. Campanato,
{\it $L^p$ Regularity and partial Hölder continuity for solutions of second order parabolic systems with strictly controlled growth},
Annali di Matematica Pura ed Applicata {\bf 128} (1981), 287--316.

\bibitem[EFV92]{EFV}
L. Escauriaza, E. B. Fabes and G. Verchota,
{\it On a regularity theorem for weak solutions to transmission problems with internal Lipschitz boundaries},
Proc. Amer. Math. Soc. {\bf 115} (1992), 1069–1076.


\bibitem[FS14]{FS14}
A. Figalli and H. Shahgholian,
{\it A general class of free boundary problems for fully nonlinear elliptic equations},
Arch. Rational Mech. Anal. {\bf 213} (2014), 269--286. 

\bibitem[FS15]{FS15}
A. Figalli and H. Shahgholian,
{\it A general class of free boundary problems for fully nonlinear parabolic equations},
Annali di Matematica {\bf 194} (2015), 1123--1134. 

\bibitem[Gia78]{Gia78}
M. Giaquinta,
{\it A counter-example to the boundary regularity of solutions to elliptic quasilinear systems},
Manuscripta Math. {\bf 24} (1978), 217--220. 

\bibitem[GM12]{GM12}
M. Giaquinta and L. Martinazzi,
{\it An Introduction to the Regularity Theory for Elliptic Systems, Harmonic Maps and Minimal Graphs},
Harmonic Maps and Minimal Graphs. Pisa, Edizioni Della Normale, 2012.
 
 \bibitem[Isa17]{Isak-2017}
V. Isakov,
Victor Inverse problems for partial differential equations. Third edition. Applied Mathematical Sciences, 127. Springer, Cham, 2017. xv+406 pp.

\bibitem[KLS1]{KLS1} 
S. Kim, K.-A. Lee and H. Shahgholian, 
{\it An elliptic free boundary arising from the jump of conductivity},
Nonlinear Anal. {\bf 161} (2017), 1--29.

\bibitem[KLS2]{KLS2} 
S. Kim, K.-A. Lee and H. Shahgholian, 
{\it Nodal sets for ``broken'' quasilinear PDEs},
Indiana Univ. Math. J. {\bf 68} (2019), 1113--1148.

\bibitem[Li17]{Li17}
Y. Li,
{\it On the $C^1$ regularity of solutions to divergence form elliptic systems with Dini-continuous coefficients},
Chin. Ann. Math. {\bf 38B} (2017), 489--496.

\bibitem[LN03]{LN}
Y. Li and L. Nirenberg,
{\it Estimates for elliptic systems from composite material},
Comm. Pure Appl. Math. {\bf 56} (2003), 892--925. 

\bibitem[SY02]{SY}
V. Sver\'ak and X. Yan,
{\it Non-Lipschitz minimisers of smooth uniformly convex variational integrals},
Proc. Natl. Acad. Sci. USA {\bf 99} (2002), 15269--15276. 



\end{thebibliography}
\end{document}